\documentclass[letterpaper,11pt,reqno]{amsart}

\makeatletter
\usepackage{amssymb}
\usepackage{latexsym}
\usepackage{amsbsy}
\usepackage{amsfonts}
\usepackage{hyperref}
\usepackage{graphicx}
\usepackage{mathtools}
\usepackage{enumerate}
\usepackage{url}
\usepackage{color}

\usepackage{tikz}

\DeclarePairedDelimiter\floor{\lfloor}{\rfloor}
\def\marginpar#1{\ignorespaces}

\DeclareMathOperator\Bin{Bin}
\DeclareMathOperator\PB{PB}
\DeclareMathOperator\Poi{Poi}
\DeclareMathOperator\Var{Var}
\DeclareMathOperator\TP{TP}
\DeclareMathOperator\ess{ess}
\DeclareMathOperator\Acc{Acc}

\DeclareMathOperator\Ber{Ber}

\newtheorem{theorem}{Theorem}[section]

\newtheorem{proposition}[theorem]{Proposition}
\newtheorem{corollary}[theorem]{Corollary}
\newtheorem{definition}[theorem]{Definition}
\newtheorem{examples}[theorem]{Examples}
\newtheorem{example}[theorem]{Example}
\newtheorem{conj}[theorem]{Conjecture}
\newtheorem{op}[theorem]{Open problem}
\numberwithin{equation}{section}
\makeatother
\begin{document}
\title[Poisson binomial]{The Poisson binomial distribution -- Old \& New}

\author[Wenpin Tang]{{Wenpin} Tang}
\address{Department of Industrial Engineering and Operations Research, UC Berkeley. Email: 
} \email{wenpintang@stat.berkeley.edu}

\author[Fengmin Tang]{{Fengmin} Tang}
\address{UCLA. Email: 
} \email{tfmin1998@ucla.edu}

\date{\today} 
\begin{abstract}
This is an expository article on the Poisson binomial distribution.
We review lesser known results and recent progress on this topic, including geometry of polynomials and distribution learning.
We also provide examples to illustrate the use of the Poisson binomial machinery. 
Some open questions of approximating rational fractions of the Poisson binomial are presented.
\end{abstract}

\maketitle



\maketitle
\section{Introduction}
\quad The binomial distribution is one of the earliest examples a college student encounters in his/her first course in probability.
It is a discrete probability distribution of a sum of independent and identically distributed (i.i.d.) Bernoulli random variables, modeling the number of occurrence of some events in repeated trials.
An integer-valued random variable $X$ is called binomial with parameters $(n,p)$, denoted as $X \sim \Bin(n,p)$, 
if $\mathbb{P}(X = k) = \binom{n}{k} p^k(1-p)^{n-k}$, $0 \le k \le n$. 
It is well known that if $n$ is large, the $\Bin(n,p)$ distribution is approximated by the Poisson distribution for small $p$'s, and is approximated by the normal distribution for larger values of $p$.
See e.g. \cite{PitmanProb} for an educational tour. 

\quad Poisson \cite{Poisson1837} considered a more general model of independent trials, which allows heterogeneity among these trials. 
Precisely, an integer-valued random variable $X$ is called Poisson binomial, and denoted as $X \sim \PB(p_1, \ldots, p_n)$ if
\begin{equation*}
X \stackrel{(d)}{=} \xi_1 + \cdots + \xi_n,
\end{equation*}
where $\xi_1, \ldots, \xi_n$ are independent Bernoulli random variables with parameters $p_1, \ldots, p_n$.
It is easily seen that the probability distribution of $X$ is 
\begin{equation}
\label{eq:PMF}
\mathbb{P}(X = k) = \sum_{A \in [n], \, |A| = k} \left( \prod_{i \in A} p_i \prod_{i \notin A} (1-p_i) \right),
\end{equation}
where the sum ranges over all subset of $[n]: = \{1, \ldots, n\}$ of size $k$. 

\quad The Poisson binomial distribution has a variety of applications such as reliability analysis \cite{BP83, HMM}, survey sampling \cite{CL97, Wang93}, finance \cite{DSW07, S99}, and engineering \cite{FA03, TA11}.
Though this topic has been studied for a long time, the literature is scattered.
For instance, the Poisson binomial distribution has different names in various contexts: P\'olya frequency (PF) distribution, strongly Rayleigh distribution, convolutions of heterogenous Bernoulli, etc.
Researchers often work on some aspects of this subject, and ignore its connections to other fields.
In late 90's, Pitman \cite{Pitman} wrote a survey on the Poisson binomial distribution with focus on probabilizing combinatorial sequences.
Due to its applications in modern technology (e.g. machine learning \cite{BPJ, Rosen18}, causal inference (Example \ref{expl})) and links to different mathematical fields (e.g. algebraic geometry, mathematical physics), we are motivated to survey recent studies on the Poisson binomial distribution.
While most results in this paper are known in some form, several pieces are new (e.g. Section \ref{sc:4}).
The aim of this paper is to provide a guide to lesser known results and recent progress of the Poisson binomial distribution, mostly post 2000.

\quad The rest of the paper is organized as follows. 
In Section \ref{sc:2}, we review distributional properties of the Poisson binomial distribution.
In Section \ref{sc:3}, various approximations of the Poisson binomial distribution are presented.
Section \ref{sc:4} is concerned with the Poisson binomial distribution and polynomials with nonnegative coefficients. 
There we discuss the problem of approximating rational fractions of Poisson binomial.
Finally in Section \ref{sc:5}, we consider some computational problems related to the Poisson binomial distribution.

\section{Distributional properties of Poisson binomial variables}
\label{sc:2}
\quad In this section, we review a few distributional properties of the Poisson binomial distribution. 
For $X \sim \PB(p_1, \ldots, p_n)$, we have
\begin{equation}
\label{eq:mv}
\mu: = \mathbb{E}X = n \bar{p} \quad \mbox{and} \quad \sigma^2 : = \Var X =  n \bar{p} (1 - \bar{p}) - \sum_{i = 1}^n (p_i - \bar{p})^2,
\end{equation}
where $\bar{p}: = \sum_{i = 1}^n p_i/n$.
It is easily seen that by keeping $\mathbb{E}X$ (or $\bar{p}$) fixed, the variance of $X$ is increasing as the set of probabilities $\{p_1, \ldots, p_n\}$ gets more homogeneous, and is maximized as $p_1 = \cdots = p_n$. 
There is a simple interpretation in survey sampling: taking samples from different communities ({\em stratified sampling}) is better than taking from the same group ({\em simple random sampling}).

\quad The above observation motivates the study of stochastic orderings for the Poisson binomial distribution.
The first result of this kind is due to Hoeffding \cite{Hoeff}, claiming that among all Poisson binomial distributions with a given mean, the binomial distribution is the most spread-out.

\begin{theorem} \cite{Hoeff} (Hoeffding's inequalities)
\label{thm:SO}
Let $X \sim \PB(p_1, \ldots, p_n)$, and $\bar{X} \sim \Bin(n, \bar{p})$. 
\begin{enumerate}
\item
There are inequalities
\begin{equation*}
\mathbb{P}(X \le k) \le \mathbb{P}(\bar{X} \le k) \quad \mbox{for } \, 0 \le k \le  n\bar{p} - 1,
\end{equation*}
and 
\begin{equation*}
\mathbb{P}(X \le k) \ge \mathbb{P}(\bar{X} \le k) \quad \mbox{for } \, n\bar{p} \le k \le n.
\end{equation*}
\item
For any convex function $g : [n] \rightarrow \mathbb{R}$ in the sense that $g(k+2) - 2 g(k+1) + g(k) > 0$, $0 \le k \le n-2$,
we have
\begin{equation*}
\mathbb{E}g(X) \le \mathbb{E}g(\bar{X}),
\end{equation*}
where the equality holds if and only if $p_1 = \cdots = p_n = \bar{p}$.
\end{enumerate}
\end{theorem}

\quad The part (2) in Theorem \ref{thm:SO} indicates that among all Poisson binomial distributions, the binomial is the largest one in convex order.
This result was extended to the multidimensional setting \cite{BZ80}, and to non-negative random variables \cite[Proposition 3.2]{BT10}.
See also \cite{NW86} for interpretations.
Next we give several applications of Hoeffding's inequalities. 

\smallskip
\begin{examples}
~
\begin{enumerate}
\item
Monotonicity of binomials. 
{\em Fix $\lambda > 0$. By taking $(p_1, \ldots, p_n) = (0, \frac{\lambda}{n-1}, \ldots, \frac{\lambda}{n-1})$, we get for $X \sim \Bin(n-1, \frac{\lambda}{n-1})$ and $X' \sim \Bin(n, \frac{\lambda}{n})$,  
\begin{equation*}
\qquad \quad \mathbb{P}(X \le k) < \mathbb{P}(X' \le k) \, \mbox{ for } \, k \le \lambda-1 \quad \mbox{and} \quad \mathbb{P}(X \le k) > \mathbb{P}(X' \le  k) \, \mbox{ for } \, k \ge \lambda.
\end{equation*}
Similarly, by taking $(p_1, \ldots, p_n) = (1, \frac{\lambda-1}{n-1}, \ldots, \frac{\lambda-1}{n-1})$, we get for $X \sim \Bin(n-1, \frac{\lambda-1}{n-1})$ and $X' \sim \Bin(n, \frac{\lambda}{n})$,
\begin{equation*}
\qquad \qquad  \mathbb{P}(X \le k-1) < \mathbb{P}(X' \le k) \, \mbox{ for } \, k \le \lambda-1 \quad \mbox{and} \quad \mathbb{P}(X \le k-1) > \mathbb{P}(X' \le  k) \, \mbox{ for } \, k \ge \lambda.
\end{equation*}
These inequalities were used in \cite{AS67} to derive the monotonicity of error in approximating the binomial distribution by a Poisson distribution. 
By letting $X \sim \Bin(n, p)$ and $Y \sim Poi(np)$, they proved $\mathbb{P}(X \le k) - \mathbb{P}(Y \le k)$ is positive if $k \le n^2 p/(n+1)$ and is negative if $k \ge np$.
The result quantifies the error of confidence levels in hypothesis testing when approximating the binomial distribution by a Poisson distribution.}
\item
Darroch's rule. 
{\em It is well known that a Poisson binomial variable has either one, or two consecutive modes. 
By an argument in the proof of Hoeffding's inequalities, Darroch \cite[Theorem 4]{Da64} showed that the mode $m$ of the Poisson binomial distribution differs from its mean $\mu$ by at most $1$.  
Precisely, he proved that
\begin{equation}
m = \left\{ \begin{array}{ccl}
k & \mbox{if} & k \le \mu < k + \frac{1}{k+2}, \\ 
k \mbox{ or } k+1 & \mbox{if} & k + \frac{1}{k+2} \le \mu \le k+1 - \frac{1}{n-k+1}, \\
k+1 & \mbox{if} & k+1 - \frac{1}{n-k+1} < \mu \le k+1.
\end{array}\right.
\end{equation}
This result was reproved in \cite{Sam65}. 
See also \cite{JS68} for a similar result concerning the median.
}
\item
Azuma-Hoeffding inequality.
{\em By the Azuma-Hoeffding inequality \cite{Azuma, Hoeff63}, for $\xi_1, \ldots, \xi_n$ independent random variables such that $0 \le \xi_i \le 1$,
\begin{equation}
\label{eq:AHgen}
\mathbb{P}\left(\sum_{i = 1}^n \xi_i \ge t \right) \le \left(\frac{\mu}{t} \right)^t \left(\frac{n-\mu}{n-t} \right)^{n-t} \quad \mbox{for } t > \mu,
\end{equation}
where $\mu: = \sum_{i = 1}^n \mathbb{E}\xi_i$.
Now we show how to derive a version of \eqref{eq:AHgen} via a Poisson binomial trick.
Given $\xi_1, \ldots, \xi_n$, let $b_i$ be independent Bernoulli with parameter $\xi_i$ and $X \sim \Bin \left(n, \frac{1}{n} \sum_{i = 1}^n  \xi_i \right)$.
We have
\begin{equation}
\label{eq:PBcond}
\mathbb{P} \left(\sum_{i = 1}^n b_i \ge t \Bigg| \sum_{i = 1}^n \xi_i \ge t \right) \le \frac{\mathbb{P} \left(\sum_{i = 1}^n b_i \ge t \right)}{\mathbb{P} \left(\sum_{i = 1}^n \xi_i \ge t \right)}.
\end{equation}
Given $\sum_{i = 1}^n \xi_i \ge t$, $\sum_{i = 1}^n b_i$ is Poisson binomial with mean greater than $t$. 
According to Hoeffding's inequality, 
\begin{equation}
\label{eq:PBtrick}
\mathbb{P} \left(\sum_{i = 1}^n b_i \ge t \Bigg| \sum_{i = 1}^n \xi_i \ge t \right) \ge \mathbb{P}\left(X \ge t \Bigg| \sum_{i = 1}^n \xi_i \ge t\right) \ge c,
\end{equation}
for some universal constant $c > 0$.
Combining \eqref{eq:PBcond} and \eqref{eq:PBtrick} yields 
$\mathbb{P} \left(\sum_{i = 1}^n \xi_i \ge t \right) \le c \mathbb{P} \left(\sum_{i = 1}^n b_i \ge t \right)$.
Note that $\sum_{i = 1}^n b_i$ is Poisson binomial with mean $\mu$. 
Applying Hoeffding's inequality to $\sum_{i = 1}^n b_i$ with bounds for binomial tails \cite{Ok58}, we get 
\begin{equation}
\label{eq:PBconineq}
\mathbb{P}\left(\sum_{i = 1}^n b_i \ge t\right) \leq  \left(\frac{\mu}{t} \right)^t \left(\frac{n-\mu}{n-t} \right)^{n-t} \quad \mbox{for } t \ge \mu + 1.
\end{equation}
As a consequence,  $\mathbb{P}\left(\sum_{i = 1}^n \xi_i \ge t \right) \le c \left(\frac{\mu}{t} \right)^t \left(\frac{n-\mu}{n-t} \right)^{n-t} $
which achieves the same rate as in \eqref{eq:AHgen} up to a constant factor.
}
\end{enumerate}
\end{examples}

\quad The original proof of Theorem \ref{thm:SO} was brute-force, and it was soon generalized by using the idea of {\em majorization} and {\em Schur convexity}. 
To proceed further, we need some vocabularies.
Let $\{x_{(1)}, \ldots, x_{(n)}\}$ be the order statistics of $\{x_1, \ldots, x_n\}$.

\begin{definition}
\label{def:maj}
The vector ${\pmb x}$ is said to majorize the vector ${\pmb y}$, denoted as ${\pmb x} \succeq {\pmb y}$, if 
\begin{equation*}
\sum_{i = 1}^k x_{(i)} \le \sum_{i = 1}^k  y_{(i)} \quad \mbox{for } \, k \le n-1 \quad \mbox{and} \quad \sum_{i = 1}^n x_{(i)} = \sum_{i = 1}^n  y_{(i)}.
\end{equation*}
\end{definition}
See \cite{MObook} for background and development on the theory of majorization and its applications. 
The following theorem gives a few lesser known variants of Hoeffding's inequalities.

\begin{theorem}
\label{thm:genHoeff}
Let $X \sim \PB(p_1, \ldots, p_n)$, $X' \sim \PB(p'_1, \ldots, p'_n)$ and $Y \sim \Bin(n,p)$.
\begin{enumerate}
\item \cite{Gleser75, Wang93}
If $(p_1, \ldots, p_n) \succeq (p'_1, \ldots, p'_n)$, then
\begin{equation*}
\mathbb{P}(X \le k) \le \mathbb{P}(X' \le k) \quad \mbox{for } \, 0 \le k \le  n\bar{p} - 2,
\end{equation*}
and 
\begin{equation*}
\mathbb{P}(X \le k) \ge \mathbb{P}(X' \le k) \quad \mbox{for } \, n\bar{p} + 2 \le k \le n.
\end{equation*}
Moreover, $\Var(X) \le \Var(X')$.
\item \cite{PP71}
If $(-\log p_1, \ldots, - \log p_n) \succeq (- \log p'_1, \ldots, - \log p'_n)$, then $X$ is stochastically larger than $X'$, i.e. $\mathbb{P}(X \ge k) \le \mathbb{P}(X' \ge k)$ for all $k$.
\item \cite{BSC02}
$X$ is stochastically larger than $Y$ if and only if $p \le \left(\prod_{i = 1}^n p_i \right)^{\frac{1}{n}}$, and $X$ is stochastically smaller than $Y$ if and only if $p \ge 1 - \left(\prod_{i = 1}^n (1- p_i )\right)^{\frac{1}{n}}$.
Consequently, if $\left(\prod_{i = 1}^n p_i \right)^{\frac{1}{n}} \ge 1 - \left(\prod_{i = 1}^n (1- p'_i )\right)^{\frac{1}{n}}$ then $X$ is stochastically larger than $X'$.
\end{enumerate}
\end{theorem}

\quad The proof of Theorem \ref{thm:genHoeff} relies on the fact that ${\pmb{x}} \succeq {\pmb y}$ implies the components of ${\pmb{x}}$ are more spread-out than those of ${\pmb{y}}$.
For example in part (1), it boils down to proving if $k \le n \bar{p}-2$, 
$\mathbb{P}(X \le k)$ is a Schur concave function in ${\pmb p}$, meaning its value increases as the components of ${\pmb p}$ are less dispersed. 
The part (3) gives a sufficient condition of stochastic orderings for the Poisson binomial distribution. 
A simple necessary and sufficient condition remains open.
See also \cite{B07, BP83, BSC04, HJ17, SO, XB11} for further results.
\section{Approximation of Poisson binomial distributions}
\label{sc:3}
\quad In this section, we discuss various approximations of the Poisson binomial distribution. 
Pitman \cite[Section 2]{Pitman} gave an excellent survey on this topic in the mid-90's. 
We complement the discussion with recent developments. 
In the sequel, $\mathcal{L}(X)$ denotes the distribution of a random variable $X$.

\smallskip
{\bf Poisson approximation}.
Le Cam \cite{LeCam} gave the first error bound for Poisson approximation of the Poisson binomial distribution.
The following theorem is an improvement of Le Cam's bound.
\begin{theorem} \cite{BH84}
Let $X \sim \PB(p_1, \ldots, p_n)$ and $\mu: = \sum_{i = 1}^n p_i$. Then
\begin{equation}
\label{eq:TVPoi}
\frac{1}{32} \min\left(1, \frac{1}{\mu}\right) \sum_{i = 1}^n p_i^2 \le d_{TV}(\mathcal{L}(X),  \Poi(\mu)) \le \frac{1 - e^{-\mu}}{2 \mu} \sum_{i = 1}^n p_i^2,
\end{equation}
where $d_{TV}(\cdot, \cdot)$ is the total variation distance.
\end{theorem}

\quad It is easily seen from \eqref{eq:TVPoi} that the Poisson approximation of the Poisson binomial is good if $\sum_{i = 1}^n p_i^2 \ll \sum_{i = 1}^n p_i$, or equivalently $\mu - \sigma^2 \ll \mu$.
There are two cases: 
\begin{itemize}
\item
For small $\mu$, the upper bound in \eqref{eq:TVPoi} is sharp.
\item
For large $\mu$, the approximation error is of order $\sum_{i = 1}^n p_i^2 / \sum_{i = 1}^n p_i$.
\end{itemize}
As pointed out in \cite{Janson94}, the constant $1/32$ in the lower bound can be improved to $1/14$. 
See \cite{BHJ} for a book-length treatment, and \cite{Roos99} for sharp bounds.
A powerful tool to study the approximation of the sum of (possibly dependent) random variables is Stein's method of exchangeable pairs, see \cite{CDM}. 
For instance, a simple proof of the upper bound in \eqref{eq:TVPoi} was given in \cite[Section 3]{CDM} via the Stein machinery. 

\quad The Poisson approximation can be viewed as a mean-matching procedure. 
The failure of the Poisson approximation is due to a lack of control in variance. 
A typical example is where all $p_i$'s are bounded away from $0$, so that $\mu$ is large and $\sum_{i = 1}^n p_i^2 / \sum_{i = 1}^n p_i$ is of constant order. 
To deal with these cases, R\"{o}llin \cite{R07} proposed a mean/variance-matching procedure. 
To present further results, we need the following definition.

\begin{definition}
An integer-valued random variable $X$ is said to be translated Poisson distributed with parameters $(\mu, \sigma^2)$, denoted as $\TP(\mu, \sigma^2)$, if 
$X - \mu + \sigma^2 + \{\mu - \sigma^2\} \sim \Poi(\sigma^2 + \{\mu - \sigma^2\})$, where $\{\cdot\}$ is the fraction part of a positive number.
\end{definition}

\quad It is easy to see that a $\TP(\mu, \sigma^2)$ random variable has mean $\mu$, and variance $\sigma^2 + \{\mu + \sigma^2\}$ which is between $\sigma^2$ and $\sigma^2 + 1$.
The following theorem gives an upper bound in total variation between a Poisson binomial variable and its translated Poisson approximation.

\begin{theorem} \cite{R07}
\label{thm:translatedP}
Let $X \sim \PB(p_1, \ldots, p_n)$, and $\mu : = \sum_{i = 1}^n p_i$ and $\sigma^2 : = \sum_{i = 1}^n p_i(1-p_i)$. 
Then
\begin{equation}
d_{TV}(\mathcal{L}(X), \TP(\mu, \sigma^2)) \le \frac{2 + \sqrt{\sum_{i = 1}^n p_i^3(1 - p_i)}}{\sigma^2}, 
\end{equation}
where $d_{TV}(\cdot, \cdot)$ is the total variation distance.
\end{theorem}

\quad Note that if all $p_i$'s are bounded away from $0$ and $1$, the approximation error is of order $1/\sqrt{n}$ which is optimal.
See \cite{Novak19} for the most up-to-date results of the Poisson approximation.
Now we give an application of translated Poisson approximation in observational studies. 

\begin{example}
\label{expl}
Sensitivity analysis. 
{\em
In matched-pair observational studies, an sensitivity analysis accesses the sensitivity of results to hidden bias.
Here we follow a modern approach of Rosenbaum \cite[Chapter 4]{Rosen02}.
Precisely, the sample consists of $n$ matched pairs and units in each pair are indexed by $i = 1,2$.
Each pair $k = 1, \ldots, n$ is matched on a set of observed covariates ${\pmb x}_{k1} = {\pmb x}_{k2}$, and only one unit in each pair receives the treatment.
Let $Z_{ki}$ be the treatment assignment, so $Z_{k1} + Z_{k2} = 1$.
Common test statistics for matched pairs are sign-score statistics of the form:
$T = \sum_{k = 1}^n d_k (c_{k1} Z_{k1} + c_{k2} Z_{k2})$,
where $d_k \ge 0$ and $c_{ki} \in \{0,1\}$.  
For simplicity, we take $d_k = 1$ and the statistics of interest are
\begin{equation}
T = \sum_{k = 1}^n (c_{k1} Z_{k1} + c_{k2} Z_{k2}), 
\end{equation}
where $c_{k1} Z_{k1} + c_{k2} Z_{k2}$ is Bernoulli distributed with parameter $p_k : = c_{k1} \pi_k + c_{k2} (1-\pi_k)$ with
$\pi_k : = \mathbb{P}(Z_{k1} = 1 | Z_{k1} + Z_{k2} = 1)$. 
So $T \sim \PB(p_1, \ldots, p_n)$.
For $1 \le k \le n$, let $\Gamma_k := \pi_k/(1 - \pi_k)$, which equals to $1$ if there is no hidden bias.

\quad The goal is to make inference on $T$ with different choices of $(\pi_1, \ldots, \pi_n)$ and understand which choices explain away the conclusion we draw from the null hypothesis (i.e. there is no hidden bias).
Thus, we are interested in the set
\begin{equation*}
\mathcal{R}(t, \alpha) : = \{(\pi_1, \ldots, \pi_n):  \mathbb{P}(T \ge t) \le \alpha\},
\end{equation*}
on the boundary of which the conclusion assuming no hidden bias is turned over.
However, direct computation of $\mathcal{R}(t, \alpha)$ seems hard. 
A routine way to solve this problem is to approximate $\mathcal{R}(t, \alpha)$ by a regular shape. 
To this end, we consider the following optimization problem:
\begin{equation} 
\label{eq:opt}
\begin{split}
& \max \Gamma, \\
& s.t. \quad \max_{{\pmb \pi} \in C_{\Gamma}} \mathbb{P}(T(\pi_1, \ldots, \pi_n) \ge t) \le \alpha,
\end{split}
\end{equation}
where $C_{\Gamma}$ is a constraint region. 
For instance, $C_{\Gamma} : = \{{\pmb{\pi}}: \frac{1}{1+\Gamma} \le \pi_k \le \frac{\Gamma}{1+\Gamma}\}$ corresponds to the worst-case sensitivity analysis.
By the translated Poisson approximation, the quantity $\max_{{\pmb \pi} \in C_{\Gamma}} \mathbb{P}(T(\pi_1, \ldots, \pi_n) \ge t)$ can be evaluated by the following problem which is easy to solve.
\begin{equation}
\begin{split}
& \min_{A \in \{0, \ldots, K\}} \min_{{\pmb \pi} \in C_{\Gamma}} \sum_{k = 0}^K \frac{\lambda^k e^{-\lambda}}{k !} \\
& s.t. \quad K = t - A, \, \lambda = \sum_{k = 1}^n p_k - A, \,  A \le \sum_{k = 1}^n p_k^2 < A + 1.
\end{split}
\end{equation}
}
\end{example}

{\bf Normal approximation}. 
The normal approximation of the Poisson binomial distribution follows from Lyapunov or Lindeberg central limit theorem, see e.g. \cite[Section 27]{Billingsley}.
Berry and Esseen independently discovered an error bound in terms of the cumulative distribution function for the normal approximation of the sum of independent random variables. 
Subsequent improvements were obtained by \cite{Pad89, Petrov65, Shi82, vB72} via Fourier analysis, and by \cite{CS01, CS05, Nea05, TN07} via Stein's method. 

\quad Let $\phi(x): = \frac{1}{\sqrt{2 \pi}} \exp \left(-x^2/2 \right)$ be the probability density function of the standard normal, and $\Phi(x): = \int_{-\infty}^x \phi(y)dy$ be its cumulative distribution function.
The following theroem provides uniform bounds for the normal approximation of Poisson binomial variables.

\begin{theorem}
Let $X \sim \PB(p_1, \ldots, p_n)$, and $\mu: = \sum_{i = 1}^n p_n$ and $\sigma^2: = \sum_{i = 1}^n p_i(1-p_i)$.
\begin{enumerate}
\item  \cite[Theorem 11.2]{Pla79}
There is a universal constant $C > 0$ such that 
\begin{equation}
\label{eq:Pla}
\max_{0 \le k \le n} \left|\mathbb{P}(X =  k) - \phi\left(\frac{k-\mu}{\sigma} \right)\right| \le \frac{C}{\sigma}.
\end{equation}
\item \cite{Shi82}
We have
\begin{equation}
\label{eq:Shi}
\max_{0 \le k \le n} \left|\mathbb{P}(X \le k) - \Phi\left(\frac{k-\mu}{\sigma} \right)\right| \le \frac{0.7915}{\sigma}.
\end{equation}
\end{enumerate}
\end{theorem}

\quad Other than uniform bounds \eqref{eq:Pla}-\eqref{eq:Shi}, several authors \cite{Bob18, Goldstein10, Rio09} studied error bounds for the normal approximation in other metrics. 
For $\mu$, $\nu$ two probability measures, consider
\begin{itemize}
\item
$L^p$ metric
\begin{equation*}
d_p(\mu, \nu) : = \left(\int_{-\infty}^{\infty} \left| \mu(-\infty, x] - \nu(-\infty, x] \right|^p dx \right)^{\frac{1}{p}}, 
\end{equation*}
\item
Wasserstein's $p$ metric
\begin{equation*}
\mathcal{W}_p(\mu, \nu) : = \inf_{\pi} \left(\int_{-\infty}^{\infty} \int_{-\infty}^{\infty} |x-y|^p \pi(dxdy) \right)^{\frac{1}{p}}, 
\end{equation*}
where the infimum runs over all probability measures $\pi$ on $\mathbb{R} \times \mathbb{R}$ with marginals $\mu$ and $\nu$.
\end{itemize}
Specializing these bounds to the Poisson binomial distribution, we get the following result.
\begin{theorem}
Let $X \sim \PB(p_1, \ldots, p_n)$, and $\mu: = \sum_{i = 1}^n p_n$ and $\sigma^2: = \sum_{i = 1}^n p_i(1-p_i)$.
\begin{enumerate}
\item \cite[Chapter V]{Petrov75}
There exists a universal constant $C > 0$ such that
\begin{equation}
\label{eq:dp}
d_p(\mathcal{L}(X), \mathcal{N}(\mu,\sigma^2)) \le \frac{C}{\sigma} \quad \mbox{for all } p \ge 1.
\end{equation}
\item \cite{Bob18, Rio09}
For each $p \ge 1$, there exists a constant $C_p > 0$ such that
\begin{equation}
\label{eq:Wp}
\mathcal{W}_p(\mathcal{L}(X), \mathcal{N}(\mu,\sigma^2)) \le \frac{C_p}{\sigma}.
\end{equation}
\end{enumerate}
\end{theorem}

\quad Goldstein \cite{Goldstein10} proved $L^p$ bound \eqref{eq:dp} for $p = 1$ with $C = 1$.
The general case follows from the inequality $d_p(\mu, \nu)^p \le d_{\infty}(\mu, \nu)^{p-1} d_1(\mu, \nu)$ together with Goldstein's $L^1$ bound and the uniform bound \eqref{eq:Shi}. 
By the Kantorovich-Rubinstein duality,  $d_1(\mu, \nu) = \mathcal{W}_1(\mu, \nu)$.
So the bound \eqref{eq:Wp} holds for $p = 1$ with $C_1 = 1$.
For general $p$, the bound \eqref{eq:Wp} is a consequence of the fact that for $Z = \sum_{i = 1}^n \xi_i$ with $\xi_i$'s independent, $\mathbb{E}\xi_i = 0$ and $\sum_{i = 1}^n \Var(\xi_i) = 1$, 
\begin{equation*}
\mathcal{W}_p(\mathcal{L}(Z), \mathcal{N}(0,1)) \le C_p \left(\sum_{i = 1}^n \mathbb{E}|Z_i|^{p+1}\right)^{\frac{1}{p}}.
\end{equation*}
This result was proved in \cite{Rio09} for $1 \le  p \le 2$, and generalized to all $p \ge 1$ in \cite{Bob18}.

\smallskip
{\bf Binomial approximation}.
The binomial approximation of the Poisson binomial is lesser known. 
The first result of this kind is due to Ehm \cite{Ehm91} who proved that for $X \sim \PB(p_1, \ldots, p_n)$,
\begin{equation}
d_{TV}(\mathcal{L}(X), \Bin(n, \mu/n)) \le \frac{1- (\mu/n)^{n+1} - (1 - \mu/n)^{n+1}}{ (n+1) (1 - \mu/n)\mu/n} \sum_{i = 1}^n (p_i - \mu/n)^2.
\end{equation}
Elm's approach was extended to a Krawtchouk expansion in \cite{Roos00}.
The advantage of the binomial approximation over the Poisson approximation is justified by the following result due to 
Choi and Xia \cite{CX02}.
\begin{theorem} 
Let $X \sim \PB(p_1, \ldots, p_n)$, and $\mu: = \sum_{i = 1}^n p_n$.
For $m \ge 1$, let $d_m: = d_{TV}(\mathcal{L}(X), \Bin(m, \mu/m))$.
Then for $m$ sufficiently large,
\begin{equation}
d_m < d_{m+1} < \cdots < d_{TV}(\mathcal{L}(X), \Poi(\mu)).
\end{equation}
\end{theorem}
See also \cite{BHJ, PRCS} for multi-parameter binomial approximations, and \cite{Skip12} for the P\'olya approximation of the Poisson binomial distribution.
\section{Poisson binomial distributions, polynomials with nonnegative coefficients and optimal transport}
\label{sc:4}

\quad In this section, we discuss aspects of the Poisson binomial distribution related to polynomials with nonnegative coefficients. 
For $X \sim \PB(p_1, \ldots, p_n)$, the probability generating function (PGF) of $X$ is 
\begin{equation}
\label{eq:PGF}
f(u): = \mathbb{E}X^u = \prod_{i=1}^n (p_i u + 1-p_i).
\end{equation}
It is easy to see that $f$ is a polynomial with all nonnegative coefficients, and all of its roots are real negative.
The story starts with the following remarkable theorem, due to Aissen, Endrei, Schoenberg and Whitney \cite{AESW, ASW}.
\begin{theorem} \cite{AESW, ASW}
\label{thm:PF}
Let $(a_0, \ldots, a_n)$ be a sequence of nonnegative real numbers, with associated generating polynomial $f(z) : = \sum_{i = 0}^n a_i z^i$.
The following conditions are equivalent:
\begin{enumerate}
\item
The polynomial $f(z)$ has only real roots.
\item
The sequence $(a_0/f(1), \ldots, a_n/f(1))$ is the probability distribution of a $\PB(p_1, \ldots, p_n)$ distribution for some $p_i$.
The real roots of $f(z)$ are $-(1-p_i)/p_i$ for $i$ with $p_i > 0$.
\item
The sequence $(a_0, \ldots, a_n)$ is a P\'olya frequency (PF) sequence, i.e. the Toeplitz matrix $(a_{j-i})_{i,j}$ is totally nonnegative. 
\end{enumerate}
\end{theorem}

\quad See \cite{Ando97} for background on total positivity. 
From a computational aspect, the condition (3) amounts to solving a system of $n(n-1)/2$ polynomial inequalities \cite{FJS, GP}.
Theorem \ref{thm:PF} justifies the alternative name `PF distribution' for the Poisson binomial distribution. 
Standard references for PF sequences are \cite{Brenti,Stanley}.
See also \cite{Pitman} for probabilistic interpretations for polynomials with only negative real roots, and \cite{Holtz} for various extensions of Theorem \ref{thm:PF} by linear algebra.

\quad  A polynomial is called stable if it has no roots with positive imaginary part, and a stable polynomial with all real coefficients is called real stable \cite{BB08, BB09}. 
In \cite{BBT09}, a discrete distribution is said to be strongly Rayleigh if its PGF is real stable. 
It was also shown that the strong Rayleigh property enjoys all virtues of negative dependence. 
The following result is a simple consequence of Theorem \ref{thm:PF}.
\begin{corollary}
A random variable $X \sim \PB(p_1, \ldots, p_n)$ for some $p_i$ if and only if $X$ is strongly Rayleigh on $\{0, \ldots, n\}$.
\end{corollary}
In the sequel, we use the terminologies `Poisson binomial' and `strongly Rayleigh' interchangeably.
Call a polynomial $f(z) = \sum_{i=0}^n a_i z^i$ with $a_i \ge 0$ strongly Rayleigh if it satisfies one of the conditions in Theorem \ref{thm:PF}.

\quad For $n \ge 5$, it is hopeless to get any `simple' necessary and sufficient condition for $f$ to be strongly Rayleigh due to Abel's impossibility theorem.
A necessary condition for $f$ to be strong Rayleigh is the Newton's inequality:
\begin{equation}
\label{Newton}
a_i^2 \ge a_{i-1} a_{i+1} \left(1 + \frac{1}{i} \right) \left( 1 + \frac{1}{n-i} \right), \quad 1 \le i \le n-2,
\end{equation}
The sequence $(a_i; \, 0 \le i \le n)$ satisfying \eqref{Newton} is also said to be ultra-logconcave \cite{Pem00}.
Consequently, $(a_i; \, 0 \le i \le n)$ is logconcave and unimodal.
A lesser known sufficient condition is given in \cite{H23, Kurtz}:
\begin{equation}
\label{Kurtzsuff}
a_i^2 > 4 a_{i-1} a_{i+1}. \quad 1 \le i \le n-2.
\end{equation}
See also \cite{Han, KV} for various generalizations.
As observed in \cite{KS}, the inequality \eqref{Kurtzsuff} cannot be improved since the sequence $(m_i; \, i \ge 0)$ defined by
$m_i: = \inf \left\{\frac{a_i^2}{a_{i-1}a_{i+1}} ; \, f \mbox{ is strong Rayleigh} \right\}$
decreases from $m_1 = 4$ to its limit approximately $3.2336$. 

\quad In recent work \cite{GLP}, the authors considered the multivariate CLT from strongly Rayleigh property. 
They raised the following question: if $X$ is a strong Rayleigh, or Poisson binomial random variable, how well can one approximate $jX/k$ for each $j, k \ge 1$ by a strong Rayleigh, or Poisson binomial random variable ? 
A good approximation combined with the Cr\'amer-Wold device proves the CLT for multivariate strongly Rayleigh variables. 
The case $j = 1$ was solved in that paper.
\begin{theorem}  \cite{GLP}
\label{001}
Let $X$ be a strongly Rayleigh random variable. Then $\floor*{\frac{X}{k}}$ is strongly Rayleigh for each $k \ge 1$, 
where $\floor{x}$ is the integer part of $x$.
\end{theorem}

\quad The key to the proof of Theorem \ref{001} is \cite[Theorem 4.3]{GLP}: For $f$ a polynomial of degree $n$ and $k \ge 1$, write
$f(z) = \sum_{j = 0}^{k-1} x^j g_j(z^k)$, 
with $g_j$ a polynomial of degree $\floor{\frac{n-j}{k}}$.
The theorem asserts that if $f$ is strongly Rayleigh, then so are $g_i$'s with interlacing roots.
In fact, the real-rootedness follows from the fact that
\begin{align*}
(a_n; \,n \ge 0)  \mbox{ is a P\'olya frequency } & \mbox{ sequence}  \Longrightarrow \\ 
&(a_{kn + j}; \, n \ge 0) \mbox{ is a P\'olya frequency sequence},
\end{align*}
for each $k \ge 1$ and $0 \le j < k$.
This result is well known, see \cite[Theorem 7]{AESW} or \cite[Theorem 3.5.4]{Brenti}.
But the root interlacing seems less obvious by P\'olya frequency sequences.

\quad A natural question is whether $\floor{jX/k}$ is strongly Rayleigh for each $j, k \ge 1$.
It turns out that $\floor{2X/3}$ can be far away from being strongly Rayleigh. 
In fact, one can prove the following theorem.
\begin{theorem}
\label{02}
Let $X \sim \Bin(3n, 1/2)$, and $z_i$ be the roots of the probability generating function of $\floor{2X/3}$. Then
\begin{equation}
\max_i \{\Im(z_i)\} \ge \sqrt{\frac{9n^2 - 9n -1}{2}},
\end{equation}
where $\Im(z)$ is the imaginary part of $z$.
\end{theorem}

\quad The reason why some roots of the PGF of $\floor{2X/3}$ have large positive imaginary parts is due to the unbalanced allocation of probability weights to even and odd numbers:
$\mathbb{P}\left(\floor*{\frac{2X}{3}} = 2k \right) = \binom{3n+1}{3k +1}$ while $\mathbb{P}\left(\floor*{\frac{2X}{3}} = 2k+1 \right) = \binom{3n}{3k +2}$.
So the Newton's inequality \eqref{Newton} is not satisfied.

\smallskip
{\bf Optimal transport}. 
For simplicity, we consider $X \sim \Bin(3n, 1/2)$.
The goal is to find a coupling $Y$ which is strongly Rayleigh on $\{0,1,\ldots, 2n\}$ such that $\sup |Y - 2X/3|$ is as small as possible. 
Now we provide a formulation of this problem via optimal transport. 
For $\mu$, $\nu$ two probability measures, define
\begin{equation}
\label{infwas}
\mathcal{W}_{\infty}(\mu, \nu): = \inf_{\gamma \in \pi(\mu, \nu)} \{ \gamma - \ess \sup  |x - y|\},
\end{equation}
where $\pi(\mu, \nu)$ is the set of couplings of $\mu$ and $\nu$. 
The metric $\mathcal{W}_{\infty}(\cdot, \cdot)$ is known as the $\infty$-Wasserstein distance, see \cite{Villani}.
A coupling $\gamma$ which achieves the infimum \eqref{infwas} is called an optimal transference plan.
By abuse of notation, write $\mathcal{W}_{\infty}(X,Y)$ for $X \sim \mu$, $Y \sim \nu$.
We want to solve the following optimization problem:
\begin{equation}
\label{OTpb}
\Acc\left(\frac{2X}{3}\right): = \inf  \left\{\mathcal{W}_{\infty}\left(\frac{2X}{3}, Y\right);  \, Y \mbox{ is strongly Rayleigh on } \{0,1,\ldots, 2n\}\right\}.
\end{equation}
Here $\Acc(2X/3)$ stands for the accuracy of strongly Rayleigh approximations to $2X/3$.
So the smaller the value of $\Acc(2X/3)$ is, the better the approximation is. 
In \cite{GLP}, it was conjectured that $\Acc(2X/3) = \mathcal{O}(1)$.
The problem \eqref{OTpb} can be divided into two stages:
\begin{enumerate}
\item
Given the distribution of $Y$, find an optimal transference plan $Y = \phi (2X/3)$ with possibly random $\phi$ . This is the Monge(-Kantorovich) problem.
\item
Find $Y$ among all strongly Rayleigh distributions on $\{0,1,\ldots, 2n\}$ which achieves the infimum of $\mathcal{W}_{\infty}\left(2X/3, Y\right)$.
\end{enumerate}

\quad It might be difficult to solve the problem \eqref{OTpb} explicitly, but one can obtain a good upper bound by constructing a suitable transference plan.
For example, the transference plan below shows that for $X \sim \Bin(9, 1/2)$, the variable $2X/3$ can be approximated by $Y \sim \Bin(6, 1/2)$ with $\mathcal{W}_{\infty}(2X/3, Y) \le 1$. 
This implies that $\Acc\left(2X/3\right) \leq 1$ for $X \sim \Bin(9,1/2)$.
In Appendix A, we compute $\Acc(2X/3)$ with $X \sim \Bin(n,1/2)$ for small $n$'s.
\begin{figure}[h]
\includegraphics[width=0.7\textwidth]{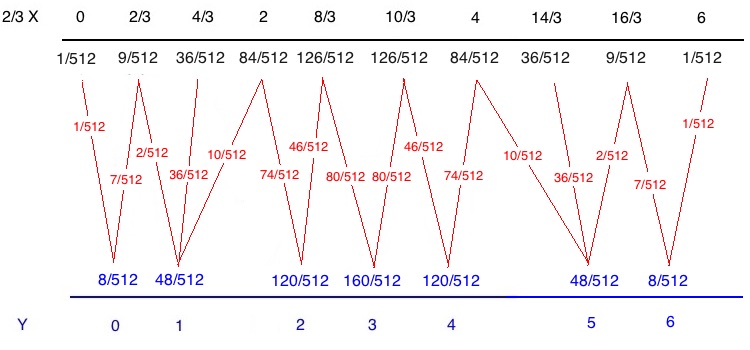}
\caption{A transference plan from $\frac{2}{3} \Bin(9,1/2)$ to $\Bin(6,1/2)$.}
\end{figure}

\quad In the part (1) of the program, one question is how well a $\Bin(2n, p)$ random variable for any $p$ can approximate $2X/3$. 
Unfortunately, the approximation is not so good as proved in the following proposition.

\begin{proposition}
\label{linearerr}
Let $X \sim \Bin(3n, 1/2)$, and $Y \sim \Bin(2n, p)$ for $0 \le p \le 1$. Then there exists $C_p  > 0$ such that
\begin{equation}
\mathcal{W}_{\infty}\left(\frac{2X}{3}, Y\right) \ge C_p n \quad \mbox{for large } n. 
\end{equation}
\end{proposition} 
\begin{proof}
The extreme cases $p = 0, 1$ are straightforward. Assume that $0 < p < 1$.
Consider transfer from $2X/3$ to $\{Y = 0\}$ with probability mass $(1-p)^{2n}$. 
By definition of $\mathcal{W}_{\infty}$,
\begin{equation*}
\mathcal{W}_{\infty}\left(\frac{2X}{3}, Y\right)  \ge \inf\left\{k; (1-p)^{2n} \le \frac{1}{2^{3n}} \sum_{i = 0}^k \binom{3n}{i}\right\}.
\end{equation*}
It is well known that for any $\lambda < 1/2$, 
$\sum_{i = o}^{3 \lambda n} \binom{3n}{i} = 2^{3n H(\lambda) + o(n)}$,
where 
$H(\lambda): = - \lambda \log_2(\lambda) - (1-\lambda) \log_2(1-\lambda)$.
It follows from standard analysis that for $p < 1 - 1/\sqrt{8}$, 
$\mathcal{W}_{\infty}\left(\frac{2X}{3}, Y\right) \ge 3 \underline{\lambda}_p n$,
where $\underline{\lambda}_p$ is the unique solution on $[0,1/2)$ to the equation $H(\lambda) = \frac{2}{3} \log_2(1 - p)+1$.
Similarly by considering transfer from $2X/3$ to $\{Y = 2n\}$ with probability mass $p^{2n}$, we get for $p >1/\sqrt{8}$,
$\mathcal{W}_{\infty}\left(\frac{2X}{3}, Y\right) \ge 3 \overline{\lambda}_p n$,
where $\overline{\lambda}_p$ is the unique solution on $[0,1/2)$ to the equation $H(\lambda) = \frac{2}{3} \log_2(p)+1$.
We take $C_p$ to be
$3  \overline{\lambda}_p$ for $p \ge 1/2$, and $3 \underline{\lambda}_p$ for $p< 1/2$.
\end{proof}

\quad The problem requires finding $(p_1, \ldots, p_{2n}) \in [0,1]^{2n}$ such that $\mathcal{W}_{\infty}\left(2X/3, \PB(p_1, \ldots, p_n \right))$ is small.
By Proposition \ref{linearerr}, the values of $p_1, \ldots, p_{2n}$ cannot be all too small or too large. 
Precisely, there exist $i \in [2n]$ such that $p_i > 1/\sqrt{8}$, and $j \in [2n]$ such that $p_j < 1 - 1/\sqrt{8}$.
This suggests to consider the equidistributed sequence $p_i = \frac{i}{2n+1}$ for $i \in [2n]$. 
By letting $Y \sim \PB(1/(2n+1), \ldots, 2n/(2n+1))$, we get
$\mathbb{E}\left(2X/3 \right) = \mathbb{E} Y = 2n$ and $\Var\left(\frac{2}{3} X \right) \sim  \Var Y\sim n/3$.
A similar argument as in Proposition \ref{linearerr} shows that
\begin{align*}
\mathcal{W}_{\infty} \left(\frac{2}{3}X, Y \right) & \ge \inf\left\{k; \, \prod_{i=1}^{2n} \frac{i}{2n+1}\leq \frac{\sum_{i=1}^k \binom{3n}{i}}{2^{3n}}  \right\} \\
&  \geq \inf \left\{k; \, \left( \frac{8}{e^2}\right)^n \leq \sum_{i=1}^k \binom{3n}{i}   \right\} = 3 \lambda_{eq} n,
\end{align*}
where $\lambda_{eq} \approx 0.0041$ is the unique solution on $[0,1/2)$ to the equation $H(\lambda) = 1 - \frac{2}{3} \log_2 (e)$.
Still the approximation is not good, but much better than the $\Bin(2n, p)$ approximation.
\begin{op}
Is there a random variable $Y \sim \PB(p_1, \ldots, p_n)$ such that $\mathcal{W}_{\infty}(2X/3, Y)$ is of order $o(n)$ ?
What is the lower bound of $\Acc(2X/3)$ ? 
\end{op}

{\bf Coefficients of Poisson binomial PGF}.
For simplicity, we take $X \sim \Bin(3n - 1, 1/2)$.
As mentioned, the most obvious approximation $\floor{2X/3}$ does not satisfy the Newton's inequality. 
It is interesting to ask the following: can we find $(a_0, \ldots, a_{2n-1}) \in \mathbb{R}_{+}^{2n}$ such that
\begin{equation}
\label{localallo}
a_{2k} + a_{2k+1} = \binom{3n-1}{3k} + \binom{3n-1}{3k+1} + \binom{3n-1}{3k+2} \quad \mbox{for } k \in [n-1],
\end{equation}
and the polynomial $P(x):=\sum_{k=1}^{2n-1} a_k x^k$ has all real roots ? 
If we are able to find such $(a_0, \ldots, a_{2n-1})$, then $\Acc(2X/3) \le 2/3$ which is a desired result. 
Note that the sequence $(a_0, \ldots, a_{2n-1})$ must satisfy the Newton's inequality and thus is unimodal.
See also \cite{Rosset} for higher order Newton's inequalities.

\quad According to \eqref{localallo}, $a_0 + a_1 = \Theta(n^2)$, meaning that $a_0 + a_1 \sim C n^2$ for some $C > 0$.
If $a_0 = \Theta(n^2)$, then the condition $a_1^2 \ge a_0 a_2$ implies that $a_2 = \mathcal{O}(n^2)$.
Further the condition $a_2^2 \ge a_1 a_3$ gives that $a_3 = \mathcal{O}(n^2)$.
Consequently, $a_2 + a_3 = \mathcal{O}(n^2)$ which contradicts the fact that $a_2 + a_3 = \Theta(n^5)$. 
So we have $a_1 = o(n)$ and $a_2 = \Theta(n^2)$. 
A similar argument shows that for any fixed $k$,
$a_{2k} = o(n^{3k+2}) \quad \mbox{and} \quad a_{2k+1} = \Theta(n^{3k+2})$.
It can be shown that $a_{k} = \Theta(n^{\frac{1+3k}{2}})$ for any fixed $k$.
But the choice for the bulk terms such as $a_{n-1}, a_{n}$ is a more subtle issue since the terms $\binom{3n}{\floor{3n/2} - 1}$, $\binom{3n}{\floor{3n/2}}$ and $\binom{3n}{\floor{3n/2}+1}$ are comparable.

\quad In Appendix A, we see that $\Acc(2X/3) = 1/3$ for $n = 1$, and $\Acc(2X/3) = 2/3$ for $n = 2$. 
Further we get,
\begin{itemize}
\item
$n = 3$: $\Acc(2X/3) = 2/3$, achieved by a strongly Rayleigh variable with PGF
$$\frac{1}{2^8}(3 + 34x + 91x^2 + 91 x^3 + 34 x^4 + 3x^5).$$
\item
$n = 4$: $\Acc(2X/3) = 2/3$, achieved by a strongly Rayleigh variable with PGF
$$\frac{1}{2^{11}}(4 + 63 x + 310 x^2 + 647 x^3 + 647 x^4 + 310 x^5 + 63 x^6 + 4x^7).$$
\item
$n = 5$: $\Acc(2X/3) = 2/3$, achieved by a strongly Rayleigh variable with PGF
\begin{align*}
\frac{1}{2^{14}}(4 +  102 x + 760.5 x^2 &+ 2606.5 x^3 + 4719 x^4 \\
&+ 4719 x^5 + 2606.5 x^6 + 760.5 x^7 + 102 x^8 + 4 x^9).
\end{align*}
\end{itemize}

\quad From small $n$ cases, we speculate there is a strongly Rayleigh polynomial $P(x)$ whose coefficients satisfy \eqref{localallo} and the symmetric/self-reciprocal condition:
\begin{equation}
\label{symmetry}
a_k = a_{2n-1-k} \quad \mbox{for } k \in [n-1].
\end{equation} 
Such polynomials are instances of {\em $\Lambda$-polynomials} \cite{Brenti90}, whose coefficients are symmetric and unimodal.
In general, for each $n \ge 2$ there exist a set of at most $n-1$ polynomials $Q_k \in \mathbb{Z}[a_0, \cdots, a_n]$ such that
the polynomial with real coefficients $P(x)$ has only real roots if and only if $Q_k \ge 0$ for each $k$.
These $Q_k$'s can be constructed as the leading coefficients of the Sturm's sequence of $P$, see e.g. \cite[Section 1.3]{Sturmfels}.
They are also the subresultants of the Sylvester matrix of $P$ and $P'$ up to sign changes.
In other words, we try to find whether the set
\begin{equation*}
S:=\{(a_0, \ldots, a_{2n-1}) \in \mathbb{R}_{+}: \eqref{localallo}, \eqref{symmetry} \mbox{ hold and } Q_k \ge 0 \mbox{ for all } k \}
\end{equation*}
is empty or not. 
The set $S$ is semi-algebraic.
According to Stengle's Positivstellensatz \cite{Stengle}, the non-emptiness of $S$ is equivalent to
\begin{equation*}
-1 \notin \mathcal{C}(Q_1, \ldots, Q_{n-1}) + \mathcal{I}\left(a_{2k} + a_{2k+1} - \binom{3n-1}{3k} - \binom{3n-1}{3k+1} - \binom{3n-1}{3k+2}, a_k - a_{2n-1-k}\right),
\end{equation*}
where $\mathcal{C}$ is the cone and $\mathcal{I}$ is the ideal.
However, the size of the polynomials $Q_k$ grows very fast, and hence exact computations become impossible. 
See also \cite{Ni, Riet} for related discussions.

\smallskip
{\bf Hurwitz stability}.
Recently, Liggett \cite{Liggett18} proved an interesting result of $\floor{2X/3}$ for $X$ a strongly Rayleigh variable.
\begin{theorem} \cite{Liggett18}
Let $X$ be a strongly Rayleigh random variable. Then the PGF of $\floor{2X/3}$ is Hurwitz stable.
That is, all its roots have negative real parts.
\end{theorem}

\quad The idea is to write the PGF of $\floor{2X/3}$ as $g_0(x^2) + x g_1(x^2)$, where $g_0$ and $g_1$ have interlacing roots.
By the Hermite-Biehler theorem \cite{Bie, Her}, such polynomials are Hurwitz stable. 
This means that the PGF of $\floor{2X/3}$ can be factorized into polynomials with positive coefficients of degrees no greater than $2$.
Thus, $\floor{2X/3}$ is a Poisson multinomial variable, that is the sum of independent random variables with values in $\{0,1,2\}$. In general, it can be shown that $\floor{jX/k}$ is expressed as
\begin{equation}
\label{interlacej}
g_0(x^j) + x g_1(x^j) + \cdots + x^{j-1}g_{j-1}(x^j),
\end{equation}
where $g_0 \ldots g_{j-1}$ have simple interlacing roots.
We conjecture the following.
\begin{conj}
Let $X$ be a strong Rayleigh random variable. Then $\floor{jX/k}$ is the sum of independent random variables with values in $\{0,1,\ldots, j\}$. 
Equivalently, the PGF of $\floor{jX/k}$ can be factorized into polynomials with positive coefficients of degrees no greater than $j$.
\end{conj}

\quad Let $P_j$ be the set of polynomials with positive coefficients which can be factorized into polynomials with positive coefficients of degrees no greater than $j$, and $Q_j$ be the set of polynomials which satisfies \eqref{interlacej}. 
From the above discussion, $P_1 = Q_1$ and $P_2 = Q_2$.
But neither implication between $P_3$ and $Q_3$ is true, as the following examples in \cite{Liggett18} show:
\begin{itemize}
\item
Let $f(z) = z^5 + z^4 + z^3 + 2z^2 + \frac{3}{2}z + \frac{1}{3}$.
The roots of $f$ are $z_1$, $\bar{z}_1$, $z_2$, $\bar{z_2}$ and $w$ with values $z_1 = 0.725 + 0.100 i$, $z_2 = 0.435 + 1.137 i$ and $w = 0.420$.
We have $(z-z_2)(z-\bar{z}_2)(z - w) = 0.623 + 1.116 z - 0.449 z^2 + z^3$, so $f \notin P_3$.
But the roots of $h_0$, $h_1$, $h_2$ are $-\frac{1}{3}$, $-\frac{3}{2}$, $-2$ respectively, so $f \in Q_3$. 
\item
Let $f(z) = (1 + z + 2z^2)(25 + z^2 + 2z^3) = 25 + 25z + 51z^2 + 3z^3 + 4z^4 + 4z^5$ , which is in $P_3$,
However, $f \notin Q_3$ since the roots of $h_0$, $h_1$, $h_2$ are $-\frac{25}{3}$, $-\frac{25}{4}$, $-\frac{51}{4}$ respectively.
\end{itemize}  
See also \cite{Briggs, XY, ZL} for discussion of positive factorizations of small degree polynomials.
\section{Computations of Poisson binomial distributions}
\label{sc:5}

\quad In this section we discuss a few computational issues of learning and computing the Poisson binomial distribution.

\smallskip
{\bf Learning the Poisson binomial distribution}.
Distribution learning is an active domain in both statistics and computer science. 
Following \cite{DL01}, given access to independent samples from an unknown distribution $P$, an error control $\epsilon > 0$ and a confidence level $\delta > 0$, a learning algorithm outputs an estimation $\widehat{P}$ such that $\mathbb{P}(d_{TV}(\widehat{P}, P) \le \epsilon) \ge 1 - \delta$. 
The performance of a learning algorithm is measured by its sample complexity and its computational complexity.

\quad For $X \sim \PB(p_1, \ldots, p_n)$, this amounts to finding a vector $(\widehat{p}_1, \ldots, \widehat{p}_n)$ defining $\widehat{X} \sim \PB(\widehat{p}_1, \ldots, \widehat{p}_n)$ such that $d_{TV}(\widehat{X}, X)$ is small with high probability. 
This is often called proper learning of Poisson binomial distributions. 
Building upon previous work \cite{Birge, DP, R07}, Daskalakis, Diakonikolas and Servedio \cite{DDS} established the following result for proper learning of Poisson binomial distributions.

\begin{theorem} \cite{DDS}
\label{thm:LearningP}
Let $X \sim \PB(p_1, \ldots, p_n)$ with unknown $p_i$'s.
There is an algorithm such that given $\epsilon, \delta > 0$, it requires
\begin{itemize}
\item
(sample complexity) $O(1/\epsilon^2) \cdot \log(1/\delta)$ independent samples from $X$,
\item
(computational complexity) $(1/\epsilon)^{O(\log^2(1/\epsilon))} \cdot O(\log n \cdot \log(1/\delta))$ operations,
\end{itemize}
to construct a vector $(\widehat{p}_1, \ldots, \widehat{p}_n)$ satisfying $\mathbb{P}(d_{TV}(\widehat{X}, X) \le \epsilon) \ge 1 - \delta$ for $\widehat{X} \sim \PB(\widehat{p}_1, \ldots, \widehat{p}_n)$.
\end{theorem}

\quad The key to the algorithm is to find subsets covering all Poisson binomial distributions, and each of these subsets is either `sparse' or `heavy'. 
Applying Birg\'e's algorithm \cite{Birge} to sparse subsets, and the translated Poisson approximation (Theorem \ref{thm:translatedP}) to heavy subsets give the desired algorithm. 
Note that the sample complexity in Theorem \ref{thm:LearningP} is nearly optimal, since $\Theta(1/\epsilon^2)$ samples are required to distinguish $\Bin(n, 1/2)$ from $\Bin(n, 1/2 + \epsilon/\sqrt{n})$ which differ by $\Theta(\epsilon)$ in total variation.
See also \cite{DKS2} for further results on learning the Poisson binomial distribution, and \cite{DDOST, DKS, DKS3} for the integer-valued distribution.

\smallskip
{\bf Computing the Poisson binomial distribution}.
Recall the probability distribution of $X \sim \PB(p_1, \ldots, p_n)$ from \eqref{eq:PMF}.
A brute-force computation of this distribution is expensive for large $n$.
Approximations in Section \ref{sc:3} are often used to estimate the probability distribution/CDF of the Poisson binomial distribution.
Here we focus on the efficient algorithms to compute exactly these distribution functions.
There are two general approaches: recursive formulas and discrete Fourier analysis. 

\quad In \cite{CL97}, the authors presented several recursive algorithms to compute \eqref{eq:PMF}. 
For $B \subset [n]$, define 
\begin{equation*}
R(k, B): = \sum_{A \subset B, \, |A| = k} \left( \prod_{i \in A} \frac{p_i}{1- p_i} \right).
\end{equation*}
So $\mathbb{P}(X = k) = R(k, [n]) \cdot \prod_{i = 1}^n (1- p_i)$.
Now the problem is to find efficient ways to compute $R(k, B)$.
Two recursive algorithms are proposed:
\begin{itemize}
\item \cite{CDL, Stein90}
For $B \subset [n]$, by letting $T(i, B): = \sum_{j \in B} \left(\frac{p_j}{1-p_j} \right)^i$, 
\begin{equation}
R(k, B) = \frac{1}{k} \sum_{i = 1}^k (-1)^{i+1} T(i, B)  R(k-i, B),
\end{equation}
\item \cite{GLR}
For $B \subset [n]$,
\begin{equation}
R(k, B) = R(k, B \setminus \{k\}) + \frac{p_k}{1-p_k} R(k-1, B \setminus \{k\}).
\end{equation}
\end{itemize}

\quad In another direction, \cite{FW10, Hong} used a Fourier approach to evaluate the probability distribution/CDF of Poisson binomial distributions.
They provided the following explicit formulas:
\begin{equation}
\label{eq:DFT}
\mathbb{P}(X = k) = \frac{1}{n+1} \sum_{j = 0}^n \exp(-i \omega k j) x_j,
\end{equation}
and
\begin{equation}
\mathbb{P}(X \le k) = \frac{1}{n+1} \sum_{j= 0}^n \frac{1 - \exp(-i\omega(k+1)j )}{1 - \exp(-i \omega j)} x_j,
\end{equation}
where $\omega: = \frac{2 \pi}{n+1}$ and $x_j: = \prod_{k = 1}^n (1 - p_k + p_k \exp(i \omega j))$.
In particular, the r.h.s of \eqref{eq:DFT} is the discrete Fourier transform of $\{x_0, \ldots, x_n\}$ which can be easily computed by Fast Fourier Transform. 
See also \cite{BZB} for a related approach.
\appendix
\section{Accuracy of $2X/3$ for small $n$}
\quad Recall the definition of $\Acc(\cdot)$ from \eqref{OTpb}. 
We compute the values of $\Acc(2X/3)$ with  $X \sim \Bin(n,1/2)$ for $1\le n \le 6$.
\begin{itemize}
\item
$n = 1$: Let $Y \sim \Ber(1/2)$, where $\Ber(p)$ is a Bernoulli variable with parameter $p$.
It is easy to see that
\begin{equation*}
\Acc(2X/3) = \mathcal{W}_{\infty}(2X/3, Y) = 1/3.
\end{equation*}
That is, the weight $\mathbb{P}(2X/3 = 0) =1/2$ is transferred to $\{Y = 0\}$, and the weight $\mathbb{P}(2X/3 = 2/3) =1/2$ is transferred to $\{Y = 1\}$.
\item
$n = 2$: Let $Y \sim \Ber(3/4)$. We have
\begin{equation*}
\Acc(2X/3) = \mathcal{W}_{\infty}(2X/3, Y) = 1/3.
\end{equation*}
So the weight $\mathbb{P}(2X/3 = 0) =1/4$ is transferred to $\{Y = 0\}$, and the weight $\mathbb{P}(2X/3 \in  \{2/3, 4/3\}) =3/4$
is transferred to $\{Y = 1\}$.
\item
$n = 3$: suppose that $\mathcal{W}_{\infty}(2X/3, Y) = 1/3$ for some integer-valued variable $Y$. 
Then the weight $\mathbb{P}(2X/3 = 0) = 1/8$ is transferred to $\{Y = 0\}$, the weight $\mathbb{P}(2X/3 \in \{2/3, 4/3\})  = 3/4$ is transferred to $\{Y = 1\}$, and the weight $\mathbb{P}(2X/3 = 2) = 1/8$ is transferred to $\{Y = 2\}$. 
The PGF of $Y$ is $1/8 + 3x/4 + x^2/8$, which has two distinct real roots $-3 \pm \sqrt{8}$.
Thus,
\begin{equation*}
\Acc(2X/3) = \mathcal{W}_{\infty}\left(2X/3, \PB \left(\frac{1}{4 + \sqrt{8}}, \frac{1}{4 - \sqrt{8}}\right) \right) = 1/3.
\end{equation*}
\item
$n = 4$: if $\mathcal{W}_{\infty}(2X/3, Y) = 1/3$ for some integer-valued $Y$, then the PGF of $Y$ is $1/16 + 10x /16  + 4x^2/16  + x^3/16 $. 
This PGF has one real root and two imaginary roots, so $Y$ cannot be strongly Rayleigh. 
There are many ways to construct a strongly Rayleigh variable $Y$ such that $\mathcal{W}_{\infty}(2X/3, Y) = 2/3$. 
For instance, the weight $\mathbb{P}(2X/3 = 0 ) = 1/16$ is transferred to $\{Y = 0\}$, the weight $\mathbb{P}(2X/3 \in \{2/3, 4/3\}) = 10/16$ is transferred to $\{Y = 1\}$ and the weight $\mathbb{P}(2X/3 \in \{2, 8/3\}) = 5/16$ is transferred to $\{Y = 2\}$.
So 
\begin{equation*}
\Acc(2X/3) = \mathcal{W}_{\infty}\left(2X/3, \PB\left(\frac{1}{2 + 2/\sqrt{5}}, \frac{1}{2 - 2/\sqrt{5}}\right) \right)  =  2/3.
\end{equation*}
In fact, we can find all strongly Rayleigh $Y$ such that $\mathcal{W}_{\infty}(2X/3, Y) = 2/3$. There are two cases:
\begin{enumerate}
\item
The range of $Y$ is $\{0,1,2\}$. Suppose $\theta_1/16$ with $\theta_1 \le 4$ of $\mathbb{P}(2X/3 = 2/3)$ is transferred to $\{Y = 1\}$, 
and $\theta_2/16$ with $\theta_2 \leq 6$ of $\mathbb{P}(2X/3 = 4/3)$ is transferred to $\{Y = 1\}$. Then the PGF of $Y$ is 
\begin{equation*}
\frac{5 - \theta_1}{16} + \frac{\theta_1 + \theta_2}{16}x + \frac{11 - \theta_2}{16} x^2.
\end{equation*}
So $Y$ is strongly Rayleigh if and only if $(\theta_1 + \theta_2)^2 \ge 4(5-\theta_1)(11-\theta_2)$. 
Figure $2$ (Left) shows the valid region of $(\theta_1, \theta_2)$.
\item
The range of $Y$ is $\{0,1,2,3\}$. Assume the same as in $(1)$, and in addition $\theta_3/16$ with $\theta_3 \le 1$ of $\mathbb{P}(2X/3 = 8/3)$ is transferred to $\{Y = 3\}$. Then the PGF of $Y$ is 
\begin{equation*}
\frac{5 - \theta_1}{16} + \frac{\theta_1 + \theta_2}{16}x + \frac{11 - \theta_2 - \theta_3}{16} x^2 + \frac{\theta_3}{16} x^3.
\end{equation*}
The discriminant of the cubic equation $ax^3 + bx^2 + cx + d = 0$ is $\Delta: = 18abcd - 4b^3d + b^2c^2 -4ac^3 -27 a^2d^2$.
According to a well known result of Cardano, the cubic equation has three real roots if and only if $\Delta \ge 0$ \cite{Cubicwiki}. Specializing to our case gives
\begin{align*}
18(5 &- \theta_1)(\theta_1 +  \theta_2)(11-\theta_2-\theta_3)\theta_3 - 4(11-\theta_2-\theta_3)^3(5 - \theta_1) \\
&+ (11 - \theta_2 -\theta_3)^2(\theta_1 + \theta_2)^2 - 4(\theta_1+\theta_2)^3 \theta_3 - 27(5 - \theta_1)^2 \theta_3^2 \ge 0.
\end{align*}
Figure $2$ (Right) shows the valid region of $(\theta_1, \theta_2, \theta_3)$.
\begin{figure}[h]
\includegraphics[width=0.75 \textwidth]{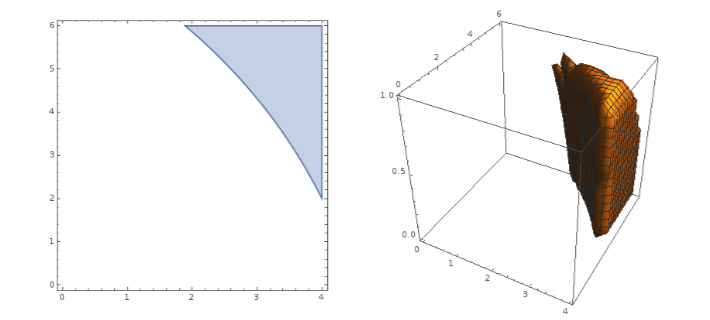}
\caption{Left: Valid region of $(\theta_1,\theta_2)$. Right: Valid region of $(\theta_1,\theta_2, \theta_3)$.}
\end{figure}
\end{enumerate}
\item
$n = 5$: a similar argument as in the case $n=4$ shows that $\mathcal{W}_{\infty}(2X/3, Y) \ne 1/3$ for each strongly Rayleigh variable $Y$. 
Again there are many ways to construct a strongly Rayleigh variable $Y$ such that $\mathcal{W}_{\infty}(2X/3, Y) = 2/3$. 
For instance, the weight $\mathbb{P}(2X/3 = 0) = 1/32$ is transferred to $\{Y = 0\}$, the weight $\mathbb{P}(2X/3 \in \{2/3, 4/3\}) = 15/32$ is transferred to $\{Y = 1\}$, the weight $\mathbb{P}(2X/3 \in \{2, 8/3\}) = 15/32$ is transferred to $\{Y = 2\}$, and the weight $\mathbb{P}(2X/3 = 10/3) = 1/32$ is transferred to $\{Y = 3\}$. The PGF of $Y$ is then
$1/32 + 15x/32 + 15 x^2/32+ x^3/32$.
It is easily seen that the coefficients of the above PGF satisfy the Hutchinson-Kurtz condition \eqref{Kurtzsuff}.
So $\Acc(2X/3) = 2/3$.
It is more difficult to find all strongly Rayleigh variables $Y$ such that $\mathcal{W}_{\infty}(2X/3, Y) = 2/3$, since the conditions for a quartic function to have all real roots are more complicated \cite{Rees}.
\item
$n = 6$: consider the transference plan in Figure 3.
\begin{figure}[h]
\includegraphics[width=0.6 \textwidth]{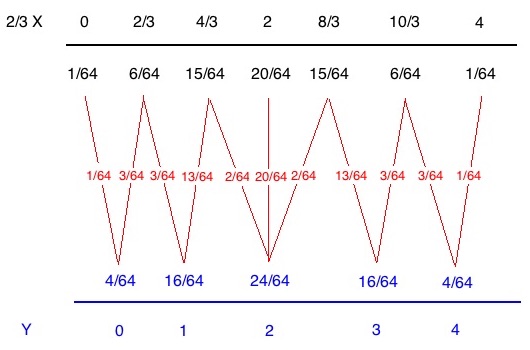}
\caption{A transference plan from $\frac{2}{3} \Bin(6, 1/2)$ to $\Bin(4,1/2)$.}
\end{figure}
It is easy to see that the PGF of $Y$ is $1/16 (1+x)^4$, so $Y \sim \Bin(4, 1/2)$ and $\Acc(2X/3) = 2/3$. 
\end{itemize}

\bigskip
{\bf Acknowledgment:} We thank Tom Liggett, Jim Pitman and Terry Tao for helpful discussions.
We thank Yuting Ye for providing Example \ref{expl}, and Tom Liggett for showing us the manuscript \cite{Liggett18}.

\bibliographystyle{plain}
\bibliography{unique}
\end{document}